\documentclass[11pt]{amsart}
\usepackage{tikz-cd}
\usepackage{amsmath,amssymb, xypic,verbatim, amscd,color}
\usepackage[margin=0.9in]{geometry}
\usepackage{graphicx}
\usepackage{colordvi}
\usepackage{mathdots}
\usepackage[colorlinks=true, pdfstartview=FitV,
linkcolor=cyan, citecolor=red, urlcolor=blue]{hyperref}
\usepackage{amsfonts,latexsym}

\usepackage{thmtools}
\usepackage{thm-restate}
\usepackage{hyperref}
\usepackage{cleveref}
\numberwithin{equation}{section}

\usepackage{amsfonts,amsmath,url, amsthm,xypic,amssymb,graphicx,color,enumerate,mathrsfs, amscd}
\usepackage{bm}

\usepackage{enumerate,bbm}
\usepackage[english]{babel}

\usepackage{ulem}
\usepackage[all]{xy}

\usepackage{thmtools}
\usepackage{thm-restate}

\usepackage{hyperref}

\def\Q{\mathbb{Q}}
\def\P{\mathbb{P}}
\def\C{\mathbb{C}}
\def\A{\mathbb{A}}
\def\Z{\mathbb{Z}}

\def\x{\times}
\def\ox{\otimes}
\def\a{\alpha}

\def\tlt0{(\tilde LT/T)_0}

\def\cO{\mathcal{O}}

\DeclareMathOperator{\cok}{cok}

\DeclareMathOperator{\im}{im}
\DeclareMathOperator{\bun}{Bun}

\newcommand{\mc}[1]{\mathcal{#1}}

\newtheorem{thm}{Theorem}[section]
\newtheorem{prop}[thm]{Proposition}
\newtheorem{lemma}[thm]{Lemma}
\newtheorem{cor}[thm]{Corollary}
\newtheorem{conj}{Conjecture}

\theoremstyle{definition}

\theoremstyle{remark}
\newtheorem{rmk}{Remark}
\theoremstyle{notation}

\newtheorem{example}{Example}

\DeclareMathOperator{\rank}{rank}
\DeclareMathOperator{\proj}{Proj\;}

\usepackage{pgf}
\usepackage{tikz}
\usetikzlibrary{arrows,automata} 
\begin{document}

\title[Monad]{
  Natural Cohomology on $\P^1 \x \P^1$
} 
 
\author{ Pablo Solis}
\address{Department of Mathematics,
 Stanford, CA}
\email{pablos.inbox@gmail.com}

\thanks{I would like to thank Ravi Vakil for many helpful discussions.
}

\begin{abstract}
A vector bundle on a projective variety has a natural cohomology if for every twist its cohomology is concentrated in a single degree. Eisenbud and Schreyer conjectured there should be vector bundles on $\P^1 \x \P^1$ with natural cohomology with respect to bundles $\cO(1,0),\cO(0,1)$ with prescribed Hilbert polynomial. We prove this conjecture. 
\end{abstract}

\keywords{ vector bundles; sheaf cohomology}

\maketitle

\tableofcontents

\section{Introduction}
A vector bundle $E$ on $\P^1 \x \P^1$ is said to have natural cohomology if for all twists $E(n,m) = E\ox \mc{O}_{\P^1 \x \P^1}(n,m)$ the cohomology $H^*(E(n,m))$ is concentrated in one degree. A conjecture \cite{EisenbudMR2810424}[pg.\;47], made by Eisenbud and Schreyer in the context of Boij-S\"oderberg theory states:
\begin{conj}[Eisenbud and Schreyer]\label{conj}
For any $P(x,y) = (x-\alpha)(y-\beta)-\gamma \in \Q[x,y]$ with $\gamma>0$ there exists a vector bundle $E$ with natural cohomology and Hilbert polynomial $\chi(E(a,b)) = \rank(E) P(a,b)$ for $\rank(E)$ sufficiently big.
\end{conj}
If $E$ satisfies conjecture \ref{conj} for a given $\alpha,\beta,\gamma$ then $E(n,m)$ satisfies conjecture \ref{conj} for $\alpha+n,\beta+m,\gamma$. Therefore we can assume $\alpha,\beta \in [0,1)_\Q:= [0,1) \cap \Q$. 

In \cite{Solis2017} I prove the conjecture provided at least one of $\a$ or $\beta$  is in $(0,1)_\Q:= (0,1) \cap \Q$; that is, $(\alpha,\beta) \in [0,1)_\Q \x (0,1)_\Q\cup (0,1)_\Q \x [0,1)_\Q$. The purpose of this paper is to cover the remaining case: $P(x,y) = x y - \gamma$. In fact this last case breaks up into $0<\gamma\leq 1$ and $\gamma >1$. Both cases are handled using the same basic strategy: monads.

A monad is a complex of vector bundles
\[
0 \to A \xrightarrow{f} B \xrightarrow{g} C \to 0
\]
such that $f$ is injective and $g$ surjective. In this case the cohomology $E = \ker(g)/\im(f)$ is a vector bundle. Monads simultaneously generalize kernels and cokernels of morphisms. Monads have been used extensively to study vector bundles on projective space for example in \cite{MR509589,MR0498569,MR564443}.  In fact Eisenbud and Schreyer produce vector bundles that satisfy conjecture \ref{conj} for specific choices of $P$. These bundles are constructed as kernels of maps
\[
\begin{array}{c}
\cO(-1,-1)^a\\
\cO(-1,0)^b
\end{array} \xrightarrow{\ \phi \ } \begin{array}{c}
\cO(0,-1)^c\\
\cO^d
\end{array}
\]
and after Eisenbud and Schreyer state their conjecture they go on to say \cite{EisenbudMR2810424}[pg.\;47]
\begin{quote}
More precisely we conjecture that these bundles can be obtained from a suitable matrix $\phi$ with entries of bidegree $(1,0)$,$(0,1)$ and $(1,1)$ only, as the bundles above. This amounts to a maximal rank conjecture for such matrices.
\end{quote}

We show in theorem \ref{thm:main} we can construct the necessary bundles from a monad whose terms are direct sums of $\cO(-1,-1),\cO(-1,0),\cO(0,-1),\cO$. But in general all terms of the monad are nonzero so the resulting bundle does not appear as a simple kernel or cokernel of a matrix.

We focus here on the cases $P(x,y) = x y - \gamma$ because the other cases are covered in \cite{Solis2017}. Let us quickly describe the approach in \cite{Solis2017}. Any vector bundle on $\P^1 \x \P^1$ can be considered as a morphism $\P^1 \to \bun_{GL_r}$ where $\bun_{GL_r}$ is the moduli stack of vector bundles on $\P^1$. This proved to be a particularly useful perspective when $(\alpha,\beta) \in [0,1)_\Q \x (0,1)_\Q\cup (0,1)_\Q \x [0,1)_\Q$ because one need only consider vector bundles which arise from morphisms $\P^1 \to \bun_{GL_r}$ whose image consists of a single point of $\bun_{GL_r}$; these bundles are called constant bundles.

Here we are using monads to cover the remaining cases. But in fact using monads is more general in the sense that the approach can be easily modified to handle all cases. We only deal with the special case $P = xy - \gamma$ to avoid having to keep up track of various different cases that arise; this is discussed after theorem \ref{thm:main} in remark \ref{rmk:all cases}. 

Here is a sketch of the argument. Let $E$ be a vector bundle on $\P^1 \x \P^1$. We say $E$ has only $H^i$ at $(n,m)$ if $H^j(E(n,m)) = 0$ for $j \neq i$. If $E$ has only $H^0,H^1$ or $H^2$ at $(n,m)$ then $E$ has natural cohomology at $(n,m)$.

A monad \eqref{eq:ABCmonad} has natural cohomology if for every twist $(n,m)$ the three terms $A,B,C$ simultaneously only have $H^0,H^1$ or $H^2$. 

\begin{itemize}
\item[(1)] The monad
\[
0\to \cO(-1,-1)^{r(\gamma - 1)} \xrightarrow{f} \begin{array}{c}
\cO(0,-1)^{r \gamma} \\ \cO(-1,0)^{r \gamma}
\end{array} \xrightarrow{g} \cO^{r \gamma} \to 0
\]
has natural cohomology and if $E = \ker(g)/\im(f)$ then $\chi(E)(x,y) = r x y - r \gamma$. 
\item[(2)] When $(a,b)$ lies in the positive quadrant then a small spectral sequence shows that
\begin{align*}
H^0(E(a,b)) &\cong \frac{\ker H^0(g)}{\im H^0(f)}\\
H^1(E(a,b)) &\cong \cok(H^0(g))\\
H^2(E(a,b)) &\cong 0
\end{align*}

\item[(3)] By a generic choice of $(f,g)$ we can guarantee that $E$ has natural cohomology at a) the positive axes $(*,0)$ and $(0,*)$ for $*>0$ and b) a finite number of twists $(i,j)$  near where $x y - \gamma$ changes sign.
\item[(4)]From the finite number of twists where $E$ has natural cohomology we can deduce that $E$ has natural cohomology everywhere except possibly the negative quadrant.
\item[(5)] By Serre duality we can use a symmetric argument to show $E$ has natural cohomology in the negative quadrant and hence showing $E$ has natural cohomology everywhere. 
\end{itemize}

The precise form of the monad  in (1) comes from fact that that only one of $\cO(-1,-1)$, $\cO(-1,0)$, $\cO(0,-1)$, $\cO$  has nonzero cohomology when we twist by $(0,0),(0,-1),(-1,0),(-1,-1)$. So if we want a monad with these terms and and a particular Hilbert polynomial $P(x,y)$ we simply compute $P(0,0)$, $P(0,-1)$, $P(-1,0)$, $P(-1,-1)$ to determine the appropriate exponents.

\subsection{Notation}
Our base field is $\C$. We denote by $\P^1$ the complex projective line. The line bundles on $\P^1$ are denoted by $\cO(n)$ for $n \in \Z$. We denote by $p \colon \P^1 \x \P^1 \to \P^1$ the first projection and $q \colon \P^1 \x \P^1 \to \P^1$ the second projection. Line bundles on $\P^1 \x \P^1$ are denoted by $\cO(n,m)$. Whether $\cO$ is a line bundle on $\P^1$ or $\P^1 \x \P^1$ will be clear from the context. The dualizing bundle on $\P^1 \x \P^1$ is $\omega = \cO(-2,-2)$.

If $F$ is a vector bundle on a scheme $X$ then $F^\vee = \hom(F,\cO_X)$. If $\omega_X$ is a dualizing bundle $X$ then the Serre dual of $F$ is $F^\star = F^\vee(\omega)$.

Finally, $r$ is always a positive integer such that $r \gamma \in \Z$.

\section{The case $\gamma \leq 1$}\label{s:gamma_leq_1}
As stated in the introduction we aim to construct vector bundles $E$ with natural cohomology such that $\chi(E(x,y)) = r x y - r \gamma$ where $r = \rank(E)$. This task breaks up into two cases: $0< \gamma \leq 1$ and $\gamma > 1$. The basic strategy to show the existence of the necessary bundles is the same in both cases. However when $\gamma \leq 1$ the argument simplifies considerably. 

We claim if $\gamma \leq 1$ then a vector bundle $E$ with natural cohomology appears as a kernel:
\begin{equation}\label{eq:gamma_less_than_1}
0 \to E = \ker(g) \to \begin{array}{c}
 \cO(0,-1)^{r \gamma} \\ \cO(-1,0)^{r \gamma} \\ \cO(-1,-1)^{r(1-\gamma)}
\end{array} \xrightarrow{g} \cO^{r \gamma} \to 0.
\end{equation}

We first show there exist choices of $g$ which are surjective maps of sheaves. Set  
\[
F_1 = \begin{array}{c}
 \cO(0,-1)^{r \gamma} \\ \cO(-1,0)^{r \gamma} \\ \cO(-1,-1)^{r(\gamma - 1)}
\end{array}, \ \ \  F_2 = \cO^{r \gamma}
\]
and let $V = \hom(F_1,F_2) \cong \A^{4 r^2 \gamma}$. 
\begin{lemma}
Assume $r\geq 2$. Let $Z \subset V$ be the subvariety of maps $g$ such that $\cok g \neq 0$. Then $\dim Z \leq (r\gamma -1)(r \gamma + r-1)+2$ and $\dim Z< 4 r^2 \gamma$.
\end{lemma}
\begin{proof}
If a map $g \in V$ is not surjective as a map of sheaves then at some point $p \in \P^1 \x \P^1$ the map on fibers is not surjective: that is,  $F_{1,p}/m_p \xrightarrow{g_p} F_{2,p}/m_p$ is not surjective. This means that $g_p$ lies in the determinental variety $Y_{r\gamma -1}$ of maps of rank $\leq r \gamma -1$. We have $\dim Y_{r \gamma -1} = (r\gamma-1)(r \gamma + r-1)$ and $\P^1 \x \P^1$ gives a 2 dimensional space of choices for $p$. Thus $\dim Z \leq (r\gamma -1)(r \gamma + r-1)+2 = r^2 \gamma^2 + r^2 \gamma -2 r \gamma - r +3$.

Moreover $r\geq 2$ so $-r+3\leq 1\leq r\gamma$, thus 
\[
r^2 \gamma^2 + r^2 \gamma -2 r \gamma +( - r +3) \leq  r^2 \gamma^2 + r^2 \gamma -2 r \gamma + r\gamma = r^2 \gamma^2 + r^2 \gamma - r \gamma
\]
Finally to finish we must show this latter quantity is smaller than $4 r^2 \gamma$: 
\begin{align*}
4 r^2 \gamma - (r^2 \gamma^2 + r^2 \gamma- r \gamma) &= 3 r^2 \gamma - r^2 \gamma^2 +r\gamma\\
& = r \gamma(3 r - r \gamma+1)\\
&>  r \gamma(3 r - r \gamma)\geq 0
\end{align*}
where the last inequality follows because $r \geq r\gamma$.
\end{proof}

Let $V^0 = V - Z$ be the subvariety of morphisms which are surjective as maps of sheaves. Then by the previous lemma $V^0$ is a non empty Zariski open subvariety.

Our argument to show there exists bundles with natural cohomology proceeds in the following steps.
\begin{itemize} 
\item[(1)] There s a Zariski open subvariety $V^\llcorner \subset V^0$ of $g$ such that $E = \ker g$ has natural cohomology along $(0,n),(n,0)$ for $n\geq 0$.
\item[(2)] There is a Zariski open subvariety $V^{(1,1)} \subset V^0$ of $g$ such that $E = \ker g$ has natural cohomology at $(1,1)$.
\item[(3)] (1) and (2) imply there is Zariski open subvariety $V^+\subset V^0$ of $g$ such that $E = \ker g$ has natural cohomology everywhere except possibly in negative quadrant: $\{(a,b)| a,b< 0\}$.
\item[(4)] By Serre duality there is a Zariski open $V^-\subset V^0$ of $g$ such that $E = \ker g$ has natural cohomology except possibly outside the positive quadrant
\item[(5)] An element of $V^+ \cap V^-$ has natural cohomology everywhere.
\end{itemize}

\begin{lemma}\label{l:along_pos_axes}
Let $g \in V^0$ let $E = \ker g$. Assume the map $H^0(g(1,0))$ and $H^0(g(0,1))$ are injective maps then $E$ has natural cohomology along the twists $(n,0)$ and $(0,n)$ for $n \geq 0$. 
\end{lemma}
\begin{proof}
The long exact sequence in cohomology gives $H^0(E) = 0, H^1(E) = H^0(\cO^{r\gamma})$ and $H^2(E) = 0$. So for any choice of $g$ the bundle $E$ has natural cohomology at $(0,0)$.

Now twist equation \eqref{eq:gamma_less_than_1} by $(0,n)$ to obtain 
\[
0 \to E(0,n) \to \begin{array}{c}
 \cO(0,n-1)^{r \gamma} \\ \cO(-1,n)^{r \gamma} \\ \cO(-1,n-1)^{r(1-\gamma)}
\end{array} \xrightarrow{g} \cO^{r \gamma}(0,n) \to 0.
\]
which in cohomology gives
\[
H^0(E(0,n)) \to H^0(\cO(0,n-1)^{r \gamma}) \xrightarrow{H^0(g(0,n))} H^0(\cO^{r \gamma}(0,n)) \to H^1(E(0,n)) \to 0.
\]
As $\chi(E(0,n))<0$ it follows that $E$ has natural cohomology at $(0,n)$ exactly when $H^0(g(0,n))$ is injective. 

To proceed further note that as the source of $g$ consists of sums of the line bundles $\cO(0,-1)$,$\cO(-1,0)$, $\cO(-1,-1)$ we can accordingly decompose  $g = \left(\begin{smallmatrix}g_1 & g_2 & g_3\end{smallmatrix}\right)$ where for example $g_1 \in \hom(\cO(0,-1)^{r \gamma}, \cO^{r\gamma})$. Then $H^0(g_2(0,n)) = H^0(g_3(0,n))=0$. Thus $H^0(g(0,n)) = H^0(g_1(0,n))$ and moreover there is a map $\cO(-1)^{r \gamma} \xrightarrow{g_1} \cO^{r \gamma}$ on $\P^1 = \proj \C[z_0,z_1]$ such that $g_1(0,n)$ is the pullback of $g_1(n)$.

Then when we apply $H^0$ this we obtain
\[
\bigoplus_{i=0}^{n-1} z_0^{n-1-i}z_1^i \C^{r \gamma} \xrightarrow{H^0(g_1(n))} \bigoplus_{i=0}^n z_0^{n-i}z_1^i\C^{r \gamma}
\]
and this map is the direct sum of the maps
\[
z_0^{n-1-i}z_1^i \C^{r \gamma} \xrightarrow{H^0(g_1(1))} 
\begin{array}{c}
z_0^{n-i}z_1^i \C^{r \gamma}\\
 \bigoplus\\
  z_0^{n-1-i}z_1^{i+1} \C^{r \gamma}
\end{array}
\]
Now if $H^0(g_1(1))$ is injective then $H^0(g_1(1))^\vee$ is surjective and so $H^0(g_1(n))^\vee$ is a direct sum of  $n$ maps $H^0(g_1(1))^\vee$. Moreover the $i$th copy of $H^0(g_1(1))^\vee$ surjects onto the $i$th summand in $H^0(\cO(n-1)^{r\gamma})^\vee$ hence $H^0(g_1(n))^\vee$ is surjective and so $H^0(g_1(n)) = H^0(g_1(0,n))$ is injective.

The same argument applies to $H^0(g(n,0)) = H^0(g_2(n,0))$. 
\end{proof}

Let $V^\llcorner \subset V^0$ denote the subset of maps $g$ such that $H^0(g(1,0))$ and $H^0(g(0,1))$ are injective.
\begin{cor}\label{c:V_llcorner non empty}
The subset $V^\llcorner$ is open and non empty and for any $g\in V^\llcorner$ the bundle $E = \ker g$ has natural cohomology along $(0,n)$ and $(n,0)$ for $n\geq 0$. 
\end{cor}
\begin{proof}
Once we prove $V^\llcorner$ is non empty then the last statement of the corollary follows from lemma \ref{l:along_pos_axes}. The condition that $H^0(g(1,0))$ and $H^0(g(0,1))$ are injective is equivalent to the maps having maximal rank which is an open condition; it is given by nonvanishing of maximal minors of $H^0(g(1,0))$ and $H^0(g(0,1))$. Thus we need to simply show existence of some $g \in V$ with $H^0(g(1,0))$ and $H^0(g(0,1))$ injective.

By symmetry it suffices to consider only $H^0(g(0,1))$ and for this we simply need to consider $\cO(-1)^{r \gamma} \xrightarrow{g_1} \cO^{r \gamma}$ on $\P^1$. Let $I_{r \gamma}$ be the $r\gamma \x r\gamma$ identity matrix. Let $l\in H^0(\cO(1))$. The particular example $g_1 = l I_{r \gamma}$ has $H^0(g_1(1))$ injective hence the corollary. 
\end{proof}

Let $V^{(1,1)}\subset V^0$ be the subset of maps $g$ such that $H^0(g(1,1))$ is surjective.
 
\begin{lemma}\label{l:(1,1)_easy_case}
Assume $r\gamma$ is even. Then subset $V^{(1,1)}$ is nonempty and open. 
\end{lemma}
\begin{proof}
As in the proof of corollary \ref{c:V_llcorner non empty} we simply need find a $g$ in $V = \hom(F_1,F_2)$ that is surjective becuase then $V^{(1,1)}$ will be the intersection of two nonempty open subset.  We again take a decomposition $g = \left(\begin{smallmatrix}g_1 & g_2 & g_3\end{smallmatrix}\right)$ and further decompose $g$ as
\[
g = \left(
\begin{array}{ccc}
g_{1,1} & 0 & g_{3,1}\\
0 & g_{2,2} & g_{3,2}
\end{array}
\right)
\] where $g_{1,1}$ and $g_{2,2}$ are $\frac{r \gamma}{2} \x r \gamma$ matrices
\[
g_{1,1} = \left(\begin{array}{ccccc}
w_0 & w_1 &  &  & \\
 &  & \ddots &  & \\
 &  &  & w_0 & w_1
\end{array}\right), \ \  g_{2,2} = \left(\begin{array}{ccccc}
z_0 & z_1 &  &  & \\
 &  & \ddots &  & \\
 &  &  & z_0 & z_1
\end{array}\right) 
\]  and $g_{3,1},g_{3,2}$ can be arbitrary. Then already $H^0(g_1(1,1))$ and $H^0(g_2(1,1))$ are enough to give a surjection onto $H^0(\cO(1,1)^{r \gamma})$.
\end{proof}
\begin{lemma}
For any $g \in V^{(1,1)}$ the bundle $E = \ker g$ has natural cohomology for at $(1,1)$.
\end{lemma}
 \begin{proof}
 This follows from the long exact sequence in cohomology associated to equation \ref{eq:gamma_less_than_1}.
 \end{proof}

Now we proceed towards step 3 in the outline sketched before lemma \ref{l:along_pos_axes}. Let $p:\P^1 \x \P^1 \to \P^1$ be the first projection and $q :\P^1 \x \P^1 \to \P^1$ the second projection.

The following is proved in \cite{Solis2017}
\begin{lemma}[Kunneth Formula]\label{l:kunneth}
Let $E$ be a vector bundle on $\P^1 \x \P^1$. Then
\begin{align*}
H^0(\P^1 \x \P^1,E) &= H^0(\P^1,p_*E)\\
 H^1(\P^1 \x \P^1,E) &=  H^1(\P^1,p_*E) \oplus H^0(\P^1,R^1p_*E)\\
 H^2(\P^1 \x \P^1,E) &= H^1(\P^1,R^1p_*E)
\end{align*}
and the same is true for pushforward along $q$.
\end{lemma}

\begin{prop}\label{p:sign_flip}
Suppose $E$ is a rank $r$ vector bundle with $\chi(E(x,y)) = r x y - r \gamma$ and $0<\gamma\leq 1$. Assume also  $E$ has natural cohomology at $(0,n)$ and $(1,n)$ where $n\geq 1$. Then $R^1p_*E(0,n) = 0$ and $p_*E(0,n)$ is a vector bundle with natural cohomology. The same statement holds for $R q_*E(n,0)$ if we know $E$ has natural cohomology at $(n,0)$ and $(n,1)$.
\end{prop}
\begin{proof}
The argument for $p$ and $q$ is the same so we prove the statement only for $p$. 

The assumption of natural cohomology at $(0,n)$ and $(1,n)$ means $E(0,n)$ has only $H^1$ and $E(1,n)$ has only $H^0$.   Comparing with lemma \ref{l:kunneth} this means $R^1p_*E(1,n)$ has no cohomology and $R^1p_*E(0,1)$ has only $H^0$ and $p_*E(0,n)$ has only $H^1$.

First we show $Rp_*E(0,n)$ is torsion free. The projection formula says
\[
Rp_*E(m,n) = Rp_*\bigg(E(0,n) \ox p^*\cO(m) \bigg) \cong \bigg(Rp_*E(0,n) \bigg) \ox \cO(m).
\]
Set $F^\bullet=Rp_*E(0,n)$. Then rephrasing we have
\begin{itemize}
\item[(a)] $H^0(E(0,n)) = H^0(F^0) = 0$ and $H^2(E(0,n))=H^1(F^1)=0$
\item[(b)] $H^1(E(1,n)) = H^1(F^0(1)) \oplus H^0(F^1(1)) = 0$ and $H^2(E(1,n))= H^1(F^1(1))=0$.
\end{itemize}
Suppose $F^0$ had torsion. Then $H^0(F^0(m)) \subset H^0(E(m,n)$ so $H^0(E(m,n))$ would always be nonzero; this contradicts that $H^0(E(0,n)) = 0$. On the other hand if $F^1$ had torsion then $H^0(F^1(m)) \subset H^1(E(m,n))$ so $H^1(E(m,n))$ would always be nonzero which contradicts that $H^1(E(1,n)) = 0$. Thus $F^\bullet$ is torsion free.

If $F^1 \neq 0$ then, since it is locally free, we must have $F^1(1) = \cO(-1)^s$, $s\neq 0$. But then $F^1$ would have nonzero $H^1$ which is a contradiction. Therefore $F = F^0 = p_*E(0,1)$.

We have established that $F$ is a vector bundle such that $F$ has only $H^1$ and $F(1)$ has only $H^0$. This is only possible if $F$ is a direct sum of $\cO(-2)$s and $\cO(-1)$s; the result is proved.
\end{proof}

Set $V^+ = V^{(1,1)} \cap V^\llcorner$; it is a nonempty Zariski open subset of $V^0$.
\begin{prop}\label{p:everything_except_negative}
Suppose $g \in V^+$ and $E = \ker g$. Then $E$ has natural cohomology at all twists in the set $\{(a,b)| a\geq 0 \mbox{ or } b \geq 0\}$.
\end{prop}
\begin{proof}
By the properties of the set $V^+$ we have that $E$ has natural cohomology along $(0,n),(n,0)$ and $(1,1)$.

We apply proposition \ref{p:sign_flip} with $n=1$. This shows that $p_*E(0,1)$ and $q_*E(1,0)$ are vector bundles with natural cohomology. In particular, $(p_*E(0,1)) \ox \cO(m) \cong Rp_*E(m,1)$ has natural cohomology so by lemma \ref{l:kunneth} we see that $E$ has natural cohomology along $(m,1)$ for any $m \in \Z$. Similarly, using $q_*E(1,0)$ we get natural cohomology along $(1,m)$.

We can now apply proposition \ref{p:sign_flip} for any $n\geq 1$. Repeating the argument of the previous paragraph with $p_*E(0,n),q_*E(n,0)$ we can establish $E$ has natural cohomology at any $(a,b)$ provided $a\geq 0$ or $b\geq 0$. 
\end{proof}

Set $E^\star = E^\vee(\omega) = E^\vee(-2,-2)$ then $H^i(E(a,b)) \cong H^{2-i}(E^\star(-a,-b))$. Thus we consider the Serre dual of equation \eqref{eq:gamma_less_than_1}:
\begin{equation}\label{eq:dual_gamma_less_than_1}
0\to \cO^{r \gamma}(-2,-2) \xrightarrow{g^\star = g^\vee(\omega)}  \begin{array}{c}
\cO^{r \gamma}(-1,-2) \\
\cO^{r \gamma}(-2,-1)\\
\cO^{r-r \gamma}(-1,-1)
\end{array} \to E^\star \to 0
\end{equation}
\begin{lemma}\label{l:tautology_(1,1)}
If $g \in V^0$ and $E^\star = \cok g^\star$ then $E^\star$ has natural cohomology at $(1,1)$.
\end{lemma}
\begin{proof}
Follows immediately from the long exact sequence in cohomology. 
\end{proof}

Let $V^- \subset V^0$ be the subset of maps $g$ such that $H^1(g^\star(0,2))$ and ,$H^1(g^\star(2,0))$ are injective. 
\begin{lemma}\label{l:V- - nonemtpy}
The subset of $V^-$ is nonempty and Zariski open. 
\end{lemma}
\begin{proof}
As in corollary \ref{c:V_llcorner non empty} it suffices to give examples of $g$ such that $H^0(g^\star(0,2))$ and $H^0(g^\star(2,0))$ are injective. For example twisting equation \ref{eq:dual_gamma_less_than_1} by $(0,2)$ we see only part that survives after applying $H^1$ is the component $\cO^{r\gamma}(-2,0) \to \cO^{r \gamma}(-2,1)$. Moreover under the Kunneth decomposition $H^1(\cO(-2,n)) \cong H^1(\cO(-2)) \ox H^0(\cO(n))$ the map $H^1(g^\star(0,2))$ is determined by applying $H^0$ to a map $\cO^{r\gamma} \to \cO^{r\gamma}(1)$. Let $l\in H^0(\cO(1))$ then the map $l I_{r \gamma}\colon \cO^{r\gamma} \to \cO^{r\gamma}(1)$ is injective after applying $H^0$. 

A symmetric argument applies to show the existence of a $g$ with $H^0(g^\star(2,0))$ injective. 
\end{proof}

The rest of the argument is more or less Serre dual the argument we have given thus far and is given in the following
\begin{prop}\label{p:serre_dual}
Let $g \in V^-$ and set $E = \ker g$.
\begin{itemize}
\item[(a)] $E$ has natural cohomology along $(n,0)$ and $(0,n)$ for $n\geq 0$.
\item[(b)] $E$ has natural cohomology at $(1,1)$
\item[(c)] $E$ has natural cohomology everywhere except possibly the positive quadrant.  
\end{itemize}
\end{prop}
\begin{proof}
Twisting by $(0,0),(1,0),(0,1)$ yields a short exact sequence where either the source or target of $g^\star$ has no cohomology so $E^\star$ will always have natural cohomology at these twists. When $n\geq 2$ then part (a) is proved using the same argument given in \ref{l:along_pos_axes} utilizing the Kunneth decomposition
\begin{align*}
H^1(\cO^{r\gamma}(n-2,-2)) \cong \begin{array}{c}
H^0(\cO(n-2)^{r \gamma})\\
\ox\\
H^1(\cO(-2)^{r\gamma})
\end{array} \xrightarrow{H^1(g^\star(n,0))} \begin{array}{c}
H^0(\cO(n-1)^{r \gamma})\\
\ox\\
H^1(\cO(-2)^{r\gamma})
\end{array}\cong H^1(\cO^{r\gamma}(n-1,-2)).
\end{align*}
Part (b) is covered by lemma \ref{l:tautology_(1,1)} and now (c) follows as in the proof of \ref{p:everything_except_negative}.
\end{proof}  

\begin{cor}\label{c:gamma_leq_1}
Conjecture \ref{conj} holds for $p = x y - \gamma$ provided $\gamma \leq 1$.
\end{cor}
\begin{proof}
We have two non empty Zariski open subsets $V^+$ and $V^-$. Let $g \in V^+\cap V^-$. Set $E = \ker g$ then by proposition \ref{p:everything_except_negative} $E$ has natural cohomology everywhere except the positive quadrant but by proposition \ref{p:serre_dual}(c) $E$ also has natural cohomology in the negative quadrant. The result is proved.
\end{proof}

\section{The case $\gamma >1$}
We focus now on the remaining cases of conjecture \ref{conj}: when $P = x y - \gamma$ with $\gamma >1$. We cannot expect to realize these cases from bundles that are the kernel of a generic surjection. However we can realize them via monads. 

Let $X$ be a scheme and suppose 
\begin{equation}\label{eq:ABCmonad}
0\to A \xrightarrow{f} B\xrightarrow{g} C \to 0
\end{equation}
is a complex of vector bundles on $X$ such that $f$ is an injective bundle map and $g$ is a surjective bundle map. We set $E(f,g) = \ker(g)/im(f)$; it is a vector bundle. We sometimes call $E = E(f,g)$ a monad vector bundle. 

A useful tool to study monads is the display which is the following commutative diagram with exact rows and columns:
\begin{equation}\label{eq:display}
\xymatrix{
0 \ar[r] & A \ar[r] \ar[d]^{=} & \ker(g) \ar[r] \ar[d] & E \ar[r] \ar[d] & 0\\
0 \ar[r] & A \ar[r]^f & B  \ar[r] \ar[d]^{g}  & \cok(f) \ar[d] \ar[r] & 0\\
            &              & C  \ar[r]^{=} &            C  &     
}
\end{equation}
An easy consequence of the display is that 
\[
\chi(E) = \chi(B) - \chi(A) - \chi(C)
\]

In particular, if $\gamma \in \Q$ and $r$ is chosen so that $r \gamma \in \Z$ then when our monad is
\begin{equation}\label{eq:monad_r_gamma}
0\to \cO(-1,-1)^{r(\gamma - 1)} \xrightarrow{f} \begin{array}{c}
\cO(0,-1)^{r \gamma} \\ \cO(-1,0)^{r \gamma}
\end{array} \xrightarrow{g} \cO^{r \gamma} \to 0
\end{equation}
we have $\chi(E(x,y)) = r x y - r \gamma$. 

\begin{rmk}
In comparison with the case $0<\gamma\leq 1$ we now have two maps to worry about. Instead of using the long exact sequence in cohomology we will utilize a spectral sequence to compute the cohomology of $E(f,g)$. A second complication is more significant: in the previous section we saw that the condition to have natural cohomology was expressed as a finite number of open condition in a vector space $V$. We will see soon that when $\gamma>1$ the condition to have natural cohomology is expressed as a finite number of locally closed conditions in two different vector spaces. The key to proving the result in this case is showing the variety defined by these locally closed conditions is non empty. 
\end{rmk}

Let $\mathcal{M}^\bullet$ denote a monad as in  \eqref{eq:ABCmonad}. We set $H^i(\mathcal{M}^\bullet) = H^i(A) \oplus H^i(B) \oplus H^i(C)$. Now let $X = \P^1 \x \P^1$. We say a monad $\mathcal{M}^\bullet$ has {\it natural cohomology} if for every twist $(n,m)$ we have at most one $i$ such that $H^i(\mathcal{M}^\bullet(n,m)) \neq 0$. Any monad made of out $\cO(-1,-1),\cO(-1,0),\cO(0,-1),\cO$ has natural cohomology. In particular the set of twists is partitioned into three regions:  where there is only $H^0,H^1$ or $H^2$.

Our first result is the following
\begin{prop}\label{p:spectralResults}
Suppose $\mathcal{M}^\bullet = 0\to A \xrightarrow{f} B\xrightarrow{g} C \to 0$ is a monad on $\P^1 \x \P^1$ with natural cohomology. Let $E = E(f,g)$ be the associated vector bundle. Then 
\begin{equation}\label{eq:spectralResults111}
\begin{aligned}
H^0(E(n,m)) &\cong \frac{\ker H^0(g)}{\im H^0(f)}   &  H^0(E(n,m)) &\cong \ker H^1(g)  & H^0(E(n,m)) &\cong 0 \\
H^1(E(n,m)) &\cong \cok(H^0(g)) \ \ ,\ \  &  H^1(E(n,m)) &\cong \frac{\ker H^1(g)}{\im H^1(f)}\ \  ,\ \  & H^1(E(n,m)) &\cong \ker H^2(f) \\
H^2(E(n,m)) &\cong 0  &  H^2(E(n,m)) &\cong \cok H^1(g) & H^2(E(n,m)) &\cong \frac{\ker H^2(g)}{\im H^2(f)} 
\end{aligned}
\end{equation}
where the first column applies when $H^0(\mathcal{M}^\bullet(n,m)) \neq 0$, the second when $H^1(\mathcal{M}^\bullet(n,m)) \neq 0$ and the third when $H^2(\mathcal{M}^\bullet(n,m)) \neq 0$.
\end{prop}

\begin{proof}
We prove this result using a spectral sequence. Let $\mc{I}^{\bullet}_A, \mc{I}^{\bullet}_C, \mc{I}^{\bullet}_E$ be injective resolutions of $A,E,C$. Recall the display of a monad \eqref{eq:display}. Use the horseshoe lemma with $ \mc{I}^{\bullet}_A$ and $\mc{I}^{\bullet}_E$ we get an injective resolution $\mc{I}^{\bullet}_{\ker(g)}$ of $\ker(g)$. Now applying the horseshoe lemma again to $\mc{I}^{\bullet}_{\ker(g)}$ and $\mc{I}^{\bullet}_{C}$ to obtain an injective resolution $\mc{I}^{\bullet}_B$ of $B$.

We conclude there is a complex of sheaves of the following form
\[
\xymatrix{
\mc{I}^{\bullet}_A \ar[r] & \mc{I}^{\bullet}_B \ar[r] & \mc{I}^{\bullet}_C\\
A \ar[r]\ar[u] & B \ar[r]\ar[u]  & C . \ar[u]
}
\]
Where $\mc{I}^{\bullet}_B \cong \mc{I}^{\bullet}_A \oplus \mc{I}^{\bullet}_E \oplus \mc{I}^{\bullet}_C$.Take $\Gamma$ so we get the following double complex
\[
\xymatrix{ \vdots & \vdots & \vdots  \\
\Gamma(\mc{I}_A^1) \ar[r]\ar[u] & \Gamma(\mc{I}_B^1) \ar[r]\ar[u] & \Gamma(\mc{I}_C^1) \ar[u]\\
\Gamma(\mc{I}_A^0) \ar[r]\ar[u] & \Gamma(\mc{I}_B^0) \ar[r]\ar[u] & \Gamma(\mc{I}_C^0) \ar[u]}
\]
Here the rows are exact except in the middle where we have
\[
\frac{\ker(\Gamma(\mc{I}_B^\bullet) \to \Gamma(\mc{I}_C^\bullet))}{\im \Gamma(\mc{I}_A^\bullet)} \cong \Gamma(\mc{I}_E^\bullet).
\]
If we compute first using the horizontal arrows we then get that  
$E_{2,hor}$ is
\[
\xymatrix{ 0 & H^2(E)  &  0  \\
0 & H^1(E)  &  0\\
0 & H^0(E) \ar[uul]  &  0 \ar[uul] }
\]
This tells us that 
\begin{align*}
H^1(Tot) &\cong H^0(E)\\
H^2(Tot) &\cong H^1(E)\\
H^3(Tot) &\cong H^2(E)
\end{align*}
where $Tot$ is the total complex. Using the vertical arrows first we get $E_{1,vert}$ is
\[
\xymatrix{ 
H^2(A) \ar[r]^{H^2(f)} & H^2(B) \ar[r]^{H^2(g)}  &  H^2(C)  \\
H^1(A) \ar[r]^{H^1(f)} & H^1(B) \ar[r]^{H^1(g)} &  H^1(C)\\
H^0(A) \ar[r]^{H^0(f)} & H^0(B)  \ar[r]^{H^0(g)} &  H^0(C)  }
\]
and because the monad has natural cohomology only one of these rows is nonzero. For example the following shows $E_{2,vert}$ when the monad has only $H^0$ or $H^1$ respectively:
\begin{equation}\label{eq:E2vert}
\xymatrix{ 
0&0&0 & 0 \ar[rrd] & 0 & 0 \\
0\ar[rrd] &0&0 &\ker H^1(f) \ar[rrd]  & \frac{\ker H^1(g)}{\im H^1(f))} &  \cok H^1(g) \\
0&\frac{\ker H^0(g)}{\im H^0(f))}&\cok H^0(g) &  0 & 0  &  0 }
\end{equation}

The three different possibilities for $E_{1,vert}$ lead to the three different computations of $H^*(E)$ in the statement of the proposition. 
\end{proof}

The problem of finding a vector bundle with natural cohomology is now essentially a problem of linear algebra: find $f,g$ so that $H^i(f),H^i(g)$ have the appropriate rank. 

Here is the sketch of the argument
\begin{itemize}
\item[(1)] For a fixed choice of $f$ there is a positive dimensional family of maps $g$ such that $E = E(f,g)$ has natural cohomology at $(1,1)$ and along $(n,0),(0,n)$ for $n \geq 0$. 
\item[(2)] A generic choice of $g$ satisfying (1) will give a bundle $E$ having natural cohomology everywhere except possibly in the negative quadrant.
\item[(3)] For a fixed choice of $g$ there is a positive dimensional family of $f$ such that the Serre dual bundle $E^
\star = E(g^\star,f^\star)$ has natural cohomology at $(2,2)$ and along $(n,0),(0,n)$ for $n\geq 0$.
\item[(4)] A generic choice of $f$ satisfying (3) will give a bundle $E$ having natural cohomology everywhere except possibly in the positive quadrant.
\item[(5)] A generic choice of $f,g$ satisfying (1),(3) will have natural cohomology everywhere. 
\end{itemize}
If $F_1 \xrightarrow{\phi} F_2$ is a map of sheaves on a scheme with a dualizing sheaf $\omega$ then we set $F_i^\star = F_i^\vee(\omega)$ and $\phi^\star\colon F_2^\star \to F_1^\star$.

The twists $(1,1)$ and $(2,2)$ play a central role because they are the first twists when both maps $f,g$ `survive' at the level of cohomology. Consider the twist of equation \eqref{eq:monad_r_gamma} by $(0,n)$ with $n\geq 0$: 
\[
0\to \cO(-1,n-1)^{r(\gamma - 1)} \xrightarrow{f} \begin{array}{c}
\cO(0,n-1)^{r \gamma} \\ \cO(-1,n)^{r \gamma}
\end{array} \xrightarrow{g} \cO^{r \gamma}(0,n) \to 0
\]
then $H^*(\cO(-1,n-1)^{r(\gamma - 1)})=0$ so the cohomology of $E$ along $(0,n)$ only depends on $H^0(g(0,n))$.

Similarly if we consider the Serre dual and twist by $(2,1)$ we obtain:
\[
0\to \cO(0,-1)^{r \gamma} \xrightarrow{g^\star(2,1)} \begin{array}{c}
\cO(1,-1)^{r \gamma} \\ \cO(0,0)^{r \gamma}
\end{array} \xrightarrow{f^\star(2,1)} \cO^{r(\gamma - 1)}(1,0) \to 0
\]
and the cohomology only depends on $f^\star(2,1)$; one can also check that for any $g^\star,f^\star$ the bundle $E^\star$ will always have natural cohomology at $(1,1)$. 

To begin the argument consider the affine spaces
\begin{align*}
W =& \hom\left(\cO(-1,-1)^{r(\gamma - 1)},\begin{array}{c}
\cO(0,-1)^{r \gamma} \\ \cO(-1,0)^{r \gamma}
\end{array} \right) \cong \A^{4 (r \gamma) r(\gamma-1)}\\
V =& \hom\left(\begin{array}{c}
\cO(0,-1)^{r \gamma} \\ \cO(-1,0)^{r \gamma}
\end{array},\cO^{r \gamma} \right) \cong \A^{4(r \gamma)^2}
\end{align*}
As before, let $V^0$ denote the set of maps which are surjective as maps of vector bundles. Similarly $W^0$ is the set of maps which are injective as maps of vector bundles; equivalently it is the set of maps such that the dual morphism is surjective. 

If we twist equation \eqref{eq:monad_r_gamma} by $(1,1)$ and apply $H^0$ we obtain
\begin{equation}\label{eq:H0(1,1)monad}
0\to H^0(\cO^{r(\gamma-1)}) \xrightarrow{H^0(f(1,1))} H^0\left(\begin{array}{c}
\cO(1,0)^{r\gamma}\\
\cO(0,1)^{r\gamma}
\end{array}\right) \xrightarrow{H^0(g(1,1))} H^0(\cO(1,1)^{r\gamma}) \to \cok H^0(g(1,1))
\end{equation}

\begin{lemma}\label{l:(1,1)_condition}
Let $E = E(f,g)$ be a monad bundle. Then $E$ has natural cohomology at $(1,1)$ if and only if equation \eqref{eq:H0(1,1)monad} is exact in the middle. Equivalently,
\begin{equation}\label{eq:(1,1)imkercok_condition}
\dim \im H^0(f(1,1)) = \dim \ker H^0(g(1,1)) = \dim \cok H^0(g(1,1)) = r(\gamma-1).
\end{equation}
\end{lemma}

\begin{proof}
If $E$ has natural cohomology at $(1,1)$ it must only have $H^1$ so by proposition \ref{p:spectralResults} we must have that equation \eqref{eq:H0(1,1)monad} must be
\[
0\to \C^{r (\gamma-1)} \xrightarrow{H^0(f(1,1))} \C^{4 r \gamma} \xrightarrow{H^0(g(1,1))} \C^{4 r \gamma} \to \C^{r(\gamma-1)} \to 0.
\]
\end{proof}

\begin{lemma}\label{l:f_iso_H0(f)}
Let $F_1 = \cO^{a}$ and $F_2 = \cO(1,0)^{b} \oplus \cO(0,1)^{b}$. The map
\[
\hom(F_1,F_2) \to \hom(H^0(F_1),H^0(F_2))
\]
sending $f$ to $H^0(f)$ is an isomorphism.
\end{lemma}
\begin{proof}
The map is injective and the vector spaces in questions have the same dimension. We can also see this directly:
\begin{align*}
\hom(F_1,F_2) &\cong \hom(\cO,F_2 )^{\oplus a}\\
&\cong \hom(H^0(\cO),H^0(F_2) )^{\oplus a}\\
&\cong \hom(H^0(F_1),H^0(F_2))
\end{align*}
\end{proof}
When $a = r(\gamma-1)$ and $b = r \gamma$ we have $f \in \hom(F_1,F_2) \cong W$. Moreover specifying $H^0(f(1,1))$ recovers $f(1,1)$ and to specify $H^0(f(1,1))$ it suffices to give $r(\gamma -1)$ elements of $H^0(F_2)$.

There are two projections
\[
H^0(\cO(1,0)^{r\gamma})\xleftarrow{p_1}H^0(F_2) \xrightarrow{p_2} H^0(\cO(0,1)^{r\gamma})
\]
and we say $v\in H^0(F_2)$ is {\it balanced} if $p_1(v)\neq 0$ and $p_2(v) \neq 0$.

Let $X_{(1,1)} \subset W^0 \x V^0$ be the subvariety of pairs $(f,g)$ such that the monad bundle $E = E(f,g)$ has natural cohomology at $(1,1)$.

\begin{prop}\label{p:g(1,1)_exists}
Let $v_1, \dotsc, v_{r(\gamma-1)} \in H^0(F_2)$ be linearly independent and balanced. Let $f \in W$ be the map whose columns are determined by the $v_i$. When $f \in W^0$ we can find a $g\in V^0$ such that the monad bundle $E = E(f,g)$ has natural cohomology at $(1,1)$. More specifically:
\begin{itemize}
\item[(a)] There is a positive dimensional linear space $L \subset V$ such that if $g \in L$ then $g \circ f = 0$.
\item[(b)] On a Zariski open set of $L$ we have that $H^0(f(1,1)),H^0(g(1,1))$ satisfy \eqref{eq:(1,1)imkercok_condition} of lemma \ref{l:(1,1)_condition}. 
\item[(c)] A  generic choice of $v_i$ will give an $f\in W^0$.
\item[(d)] The intersection $L\cap V^0$ is non empty. 
\item[(e)] The variety $X_{(1,1)}$ is positive dimensional. 
\end{itemize}
\end{prop} 

\begin{proof}
For $g$ to satisfy the necessary condition that $g \circ f = 0$ it must be the case that $g(v_i)=0$. A general $g$ has the form
\[
g = \left(\begin{array}{cccccc}
a^0_{1,1}w_0+a^1_{1,1}w_1 & \dotsb & a^0_{1,r \gamma}w_0+a^1_{1,r \gamma}w_1 & b^0_{1,1}z_0+b^1_{1,1}z_1 & \dotsb & b^0_{1,r \gamma}w_0+b^1_{1,r \gamma}w_1 \\
\vdots & \ddots & \vdots & \vdots & \ddots & \vdots \\
a^0_{r\gamma,1}w_0+a^1_{r \gamma,1}w_1 & \dotsb & a^0_{r\gamma,r\gamma}w_0+a^1_{r\gamma,r\gamma}w_1 & b^0_{r\gamma,1}w_0+b^1_{r \gamma,1}w_1 & \dotsb & b^0_{r\gamma,r\gamma}z_0+b^1_{r\gamma,r\gamma}z_1
\end{array}\right)
\]
The variables $a_{i,j}^k,b_{i',j'}^{k'}$ give coordinates on $V$.
The vector $g(v_i)$ has $r \gamma$ entires and each entry is a linear combination of $z_0w_0,z_0w_1,z_1w_0,z_1w_1$ whose coefficients are {\it linear} forms in the variables $a_{i,j}^k,b_{i',j'}^{k'}$. Thus the condition $g(v_i) = 0$ is the intersection of $4 r \gamma$ hyperplanes in $V$. Therefore the condition that $g\circ f = 0$ is the intersection of $4 r\gamma(r \gamma - r)$ hyperplanes in $V$. That is, the linear space of maps $g$ such that $g \circ f = 0$ has codimension at most $4 r\gamma(r \gamma - r)$. Recall $\dim V = 4 (r \gamma)^2$ so, when the intersection is transverse, there is at least a $4r^2\gamma$ dimensional linear space  $L$ of $g$ such that $g \circ f = 0$. This proves (a)

The condition that the $v_i$ are balanced is a convenient condition to rule out some bad choices. For example if $v = (z_0\ 0 \dotsc 0 )^\intercal$ then $g(v) = 0$ implies also that $g(v')=0$ where $v = (a z_0+b z_1\ 0 \dotsc 0 )^\intercal$. When $v_i$ are balanced and independent then the condition $g(v_i) = 0$ will only increase $\ker H^0(g(1,1))$ by $1$. Moreover if we take $g_{gen}$ to be the generic point of $L$ then the $v_i$ will be the only element of $\ker H^0(g(1,1))$. Thus on a Zariski open set $L^{(1,1)}$ of $L$ we have $\ker H^0(g(1,1)) =\im H^0(f(1,1))$; this ensures \eqref{eq:(1,1)imkercok_condition} holds hence (b).

A choice of the $v_i$ is by lemma \ref{l:f_iso_H0(f)} the same as choosing a point in $W$ and the choices in $W^0$ represents a Zariski open set of choices, hence (c).

To prove (d) we will show that given a point $p\in \P^1 \x \P^1$ there is a nonempty Zariski open set $L_p \subset L$ such that $g$ is surjective at $p$. Indeed, localizing at $p$ we get a map $\cO_p^{2 r\gamma}\xrightarrow{g_p} \cO_p^{r\gamma}$. Then $g_p$ being surjective is given by the nonvanishing of maximal minors of the  matrix $g_p$. Thus $g_p$ is surjective away from the vanishing of a finite number of polynomials in the coordinates of $L$. Now $L$ is covered by the open sets $L_p$ and as $L$ is quasicompact a finite number $L_{p_1}, \dotsc, L_{p_n}$ will do. Then $\bigcap_i L_{p_i}$ will satisfy (d).

Choosing $f \in W^0$ and $g \in L^{(1,1)} \cap \bigcap_i L_{p_i}$ then we can form the monad bundle $E = E(f,g)$ and it will natural cohomology at $(1,1)$ by lemma \ref{l:(1,1)_condition}. In building $E$ we made choices with continuous moduli hence (e). 
\end{proof}

We've presented a proof of existence but in practice we can produce actual examples.

\begin{example}
Let $p = x y - 2$ and set $r = 2$. The form of our monad is 
\[
0 \to \cO(-1,-1)^2 \xrightarrow{f} \begin{array}{c}
\cO(0,-1)^4 \\
\cO(-1,0)^4\\
\end{array} \xrightarrow{g} \cO^4 \to 0
\]
We begin by choosing $f$ arbitrarily
\[
\left(\begin{array}{cccccccc}
z_0 & 2z_0+3z_1 & 5z_0+7z_1 & 8z_0+9z_1 & 9w_0+w_1 & 3w_0+8w_1 & 2w_0+5w_1 & 7w_0+6w_1\\
z_0+7z_1 & 2z_0+9z_1 & 8z_0+5z_1 & 6z_0+z_1 & 3w_0+5w_1 & 7w_0+2w_1 & w_0+11w_1 & w_0
\end{array}\right)^\intercal
\]
The first column of $f$ is in the kernel of
\[
g = \left(\begin{array}{cccccc}
a^0_{1,1}w_0+a^1_{1,1}w_1 & \dotsb & a^0_{1,4}w_0+a^1_{1,4}w_1  & b^0_{1,1}z_0+b^1_{1,1}z_1 & \dotsb & b^0_{1,4}z_0+b^1_{1,4}z_1 \\
\vdots & \ddots & \vdots & \vdots & \ddots & \vdots\\
 a^0_{4,1}w_0+a^1_{4,1}w_1 & \dotsb  & a^0_{4,4}w_0+a^1_{r\gamma,r\gamma}w_1 & b^0_{4,1}z_0+b^1_{4,1}z_1 & \dotsb & b^0_{4,4}z_0+b^1_{4,4}z_1
\end{array}\right)
\]
when $16$ linear conditions are satisfied. The first of these conditions is the coefficient of
\[
\left(\begin{array}{c}
z_0w_0\\
0\\
0\\
0
\end{array}\right)
\]
is zero and this condition is
\[
a_{11}^0+2a_{12}^0+5a_{13}^0+8a_{14}^0+9b_{11}^0+3b_{12}^0+2b_{13}^0+7b_{14}^0=0
\]
Then the linear space $L$ of matrices $g$ satisfying $g \circ f = 0$ is $32$ dimensional. One can find a basis for $L$ and choose a general element. One such point is $g = \left(\begin{smallmatrix} g_1 & g_2\end{smallmatrix}\right)$ where
\begin{align*}
g_1 &= \left(\begin{array}{cccccccc}
-18860w_0-19145w_1 & \frac{26215}{2}w_0+\frac{34705}{2}w_1 & 3120w_0-4385w_1 & \frac{-16725}{2}w_0-\frac{8455}{2}w_1 \\
-4110w_0+10690w_1 & -1258w_0-12466w_1 & 4758w_0+1916w_1 & -7433w_0-2296w_1 \\
-13845w_0-4450w_1 & 8830w_0+1476w_1 & 2015w_0-3506w_1 & -6025w_0-1029w_1\\
-3035w_0+850w_1 & -711w_0-2793w_1 & 1921w_0-302w_1 & -4756w_0-2803w_1
\end{array}\right)\\
g_2 &= \left(\begin{array}{cccccccc}
 1880z_0 & 940z_0+705z_1 & 3055z_0+235z_1 & 2585z_0+1645z_1\\
 2350z_0+940z_1 & 470z_0+2350z_1 & 1880z_1 & 2820z_0+2585z_1\\
2115z_0 & 940z_0+1410z_1 & 2115z_0+3055z_1 & 1175z_0+470z_1\\
1645z_0+1175z_1 & 1410z_0+2350z_1 & 1175z_0+1175z_1 & 1645z_0+1645z_1
\end{array}\right)
\end{align*}
One can also check that these maps $f,g$ belong to $W^0,V^0$ respectively. 
\end{example}

We now prove the symmetric result applied to $f^\star, g^\star$ at the twist $(2,2)$:
\[
0\to \cO^{r \gamma} \xrightarrow{g^\star(2,2)} \begin{array}{c}
\cO(1,0)^{r \gamma} \\ \cO(0,1)^{r \gamma}
\end{array} \xrightarrow{f^\star(2,2)} \cO^{r(\gamma - 1)}(1,1) \to 0
\]
The argument is completely symmetric to what's already been stated so we collect all the dual statements in a single proposition. If $A,B$ are locally free sheaves and $U = \hom(A,B)$ then we denote $U^\intercal = \hom(B^\vee,A^\vee)$; it is isomorphic to $U$. Let $X^\star_{(2,2)} \subset (V^0)^\intercal \x (W^0)^\intercal$ be the subvariety of pairs $(g^\star, f^\star)$ such that $E^\star =E(g^\star, f^\star)$ has natural cohomology at $(2,2)$.

\begin{prop}\label{p:f_star(1,1)_exists}
Let $F_2 = \cO(1,0)^{r\gamma} \oplus \cO(0,1)^{r\gamma}$ and $f^\star \in W^\intercal$,$g^\star \in V^\intercal$.
\begin{itemize}
\item[(1)] Let $E^\star = E(g^\star,f^\star)$ be a monad bundle for a monad that is Serre dual to \eqref{eq:monad_r_gamma}. When $\gamma \leq 4$ then $E^\star$ has natural cohomology at $(2,2)$ if and only if $H^0(f^\star(2,2))$ is surjective. When $\gamma>4$ then $E^\star$ has natural cohomology at $(2,2)$ if and only if
\begin{equation}\label{eq:(2,2)imkercok_condition}
\dim \im H^0(g^\star(2,2)) = \dim \ker H^0(f^\star(2,2)) = r \gamma
\end{equation}
\item[(2)] If $\gamma \leq 4$ then a there is a nonempty Zariski open subset $(W^{(2,2)})^\intercal \subset (W^0)^\intercal$ such that for $f \in (W^{(2,2)})^\intercal$ we have $H^0(f^\star(2,2))$ is surjective. Taking $r \gamma$ elements of $\ker H^0(f^\star(2,2))$ gives the columns of a matrix $g^\star \in V^\intercal$. A general choice of $g^\star$ lies in $(V^0)^\intercal$.
\item[(3)]  Assume $\gamma >4$. Let $v_1, \dotsc, v_{r\gamma-4r} \in H^0(F_2)$ be linearly independent and balanced.
\begin{itemize}
\item[(a)] There is a positive dimensional linear space $L \subset W^\intercal$ such that if $f^\star \in L$ then $f^\star(v_i) = 0$ for all $i$.
\item[(b)] On a Zariski open set of $L$ we can find a $g^\star\in V^\intercal$ such that $v_i \in \im g^\star$, $f^\star \circ g^\star = 0$ and $H^0(f^\star(2,2)),H^0(g^\star(2,2))$ satisfy \eqref{eq:(2,2)imkercok_condition}.
\item[(c)] If the $v_i$ are generic we can moreover assume $g^\star \in (V^0)^\intercal$
\item[(d)] The intersection $L\cap (W^0)^\intercal$ is non empty.
\item[(e)] The variety $X^\star_{(2,2)}$ is positive dimensional. 
\end{itemize}
\end{itemize}
\end{prop}
\begin{proof}
Part (1) is proved exactly as is lemma \ref{l:(1,1)_condition} except that now we get two cases depending on the sign of $4 r - r \gamma$. When $4 r - r \gamma>0$ the bundle $E^\star(2,2)$ should only have $H^0$ which by proposition \ref{p:spectralResults} means that $H^0(f^\star(2,2))$ is surjective. When $4 r - r \gamma<0$ the bundle $E^\star(2,2)$ should only have $H^1$ which is equivalent to \eqref{eq:(2,2)imkercok_condition}.

For (2) by a explicit example as in lemma \ref{l:(1,1)_easy_case} we can find a choice of $f^\star(2,2)$ which has maximal rank. Therefore there is a nonempty open subset $W' \subset W^\intercal$ such that $H^0(f^\star(2,2))$ has maximal rank whenever $f^\star \in W'$. We take $(W^{(2,2)})^\intercal = W' \cap (W^0)^\intercal$. We can dualize $f^\star$ to get a map $f$ with columns $v_1, \dotsc, v_{r (\gamma -1)}$ and apply proposition \ref{p:g(1,1)_exists} to see that we can find a $g \in V^0$ and hence a $g^\star \in (V^0)^\intercal$.

The proof of (3) is the same as the proof of proposition \ref{p:g(1,1)_exists} except for some minor details. As mentioned above, a generic choice $f^\star$ will have $H^0(f^\star(2,2))$ surjective and so before imposing any conditions there will already by $4 r$ elements in $\ker H^0(f^\star(2,2))$. Therefore we consider the condition that $r \gamma - 4r$ vectors lie in $\ker H^0(f^\star(2,2))$.

To make this precise let $f^\star_{gen}$ be the generic point of $(W)^\intercal$. Let $K$ be the residue field of $f^\star_{gen}$ then we can consider $f^\star_{gen}$ as a matrix with entries in the field $K$. The conditions $f^\star_{gen}(v_i)$ amount to $(r \gamma - 4 r)4r(\gamma-1) = 4r^2(\gamma-4)(\gamma -1)$ conditions. This is less then $\dim W^\intercal = 4 r^2\gamma(\gamma - 1)$ so the common intersections $L$ of all these conditions is still positive dimensional. This proves (a).

As mentioned above, $H^0(f^\star_{gen}(2,2))$ is surjective hence $\dim \ker H^0(f^\star_{gen}(2,2)) = 4r$ and we can choose a basis $k_1, \dotsc, k_{4r}$. Then the condition that for any $p \in L$ the vectors $k_i(p)$ don't like in the span of the $v_i$ is an open condition hence taking the $v_i$ and $k_i(p)$ together give us the columns of the desired matrix $g^\star$ hence (b).

Parts (c), (d) and (e) are proved exactly as in proposition \ref{p:g(1,1)_exists}. 
\end{proof}

Let $X^{(2,2)}_{(1,1)} \subset W^0 \x V^0$ denote the variety of pairs $(f,g)$ such the monad bundle $E = E(f,g)$ has natural cohomology at $(1,1)$ and the Serre dual bundle $E^\star = E(g^\star,f^\star)$ has natural cohomology at $(2,2)$.

\begin{cor}
The variety $X^{(2,2)}_{(1,1)}$ is positive dimensional. 
\end{cor}
\begin{proof}
Let $\pi \colon X^\star_{(2,2)} \to W^\intercal$ be the projection. Choose a component $Z \subset X^\star_{(2,2)}$ of maximal dimension let $K$ be the residue field of the generic point $\zeta \in Z$. Set $f^\star_\zeta = \pi(\zeta)$; then $f^\star_\zeta$ can be considered as a $r(\gamma-1) \x 2 r \gamma$ matrix with entries in the field $K$. Let $f_\zeta$ be its transpose and let $v_1, \dotsc, v_{r(\gamma-1)}$ be the columns of $f_\zeta$. Apply proposition \ref{p:g(1,1)_exists} to the $v_i$ to produce a positive dimensional variety $L_{(1,1)}$ such that $f_\zeta \x L_{(1,1)} \subset X_{(1,1)}(K)$. 

Let $f$ be a specialization of $f_\zeta$ in $X^\star_{(2,2)}$ and $g\in L_{(1,1)}$ then $(f,g) \in X_{(1,1)}^{(2,2)}$ when $H^0(g^\star(2,2))$ has maximal rank which is an open condition on $L_{(1,1)}$. This shows there is a Zariski open subset $U\subset \pi(Z) \x L_{(1,1)}$ such that $U \subset X^{(2,2)}_{(1,1)}$ and $U$ is non empty because it contains the generic point of $f_\zeta \x L_{(1,1)}$. Hence $X_{(1,1)}^{(2,2)}$ contains a positive dimensional variety. 
\end{proof}

Over $\P^1 \x \P^1 \x X_{(1,1)}^{(2,2)}$ we have a universal monad bundle $E^{univ}$ such that 
\[
E^{univ}|_{\P^1 \x \P^1 \x (f,g)} \cong E(f,g)
\]
Let $Z_{(1,1)}^{(2,2)} \subset X_{(1,1)}^{(2,2)}$ be an irreducible component of maximal dimension.
\begin{prop}\label{p:nat_cohom_at_(a,b)}
Assume $(a,b) = (1,0),(0,1),(0,-1),(-1,0)$ or $a >0,b>0$ or $a<0,b<0$. Then there is a nonemtpy Zariski open subset $U_{a,b} \subset Z_{(1,1)}^{(2,2)}$ such that $E^{univ}(a,b)$ has natural cohomology on every fiber over $U_{(a,b)}$.
\end{prop}
\begin{proof}
In all cases the condition that $E^{univ}$ has natural cohomology at $(a,b)$ is an open condition which is given by the nonvanishing of maximal minors of $H^0(g(a,b)),H^0(f(a,b))$ or their Serre duals. It will be a nonempty condition since it will always holds at the generic point of $Z^{(2,2)}_{(1,1)}$. We will prove the case $(a,b) = (1,0)$ and $a>0,b>0$; the latter case breaks up into $r a b - r \gamma \leq 0$ and $r a b - \gamma >0$. All other cases are completely analogous. 

Twisting by $(1,0)$ yields 
\[
0 \to \cO(0,-1)^{r(\gamma -1)} \xrightarrow{f} \begin{array}{c}
\cO(1,-1)^{r\gamma} \\
\cO(0,0)^{r \gamma}\\
\end{array} \xrightarrow{g} \cO^{r \gamma}(1,0) \to 0
\]
In this case $H^0(f(1,0)) = 0$ and by proposition \ref{p:spectralResults} a monad of this form has natural cohomology at $(1,0)$ when $H^0(g(1,0))$ is injective, equivalently $H^0(g(1,0))$ has maximal rank $r \gamma$ which is an open condition.

 Twisting by $(a,b)$ and applying $H^0$ yields 
\[
0\to \C^{r(\gamma-1)a b}\xrightarrow{H^0(f(a,b))} \begin{array}{c}
\C^{r\gamma (a+1)b}\\
\C^{r \gamma a(b+1)}
\end{array} \xrightarrow{H^0(g(a,b))} \C^{(a+1)(b+1)r \gamma} \to \cok H^0(g(a,b))
\]
Assume $a>0,b>0$ and $r a b - r \gamma\leq 0$.
Then by proposition \ref{p:spectralResults} $E^{univ}$ has natural cohomology at $(a,b)$ if 
\[
\dim \ker H^0(g(a,b)) = r(\gamma-1)a b
\]
Recall $E^{univ}$ has natural cohomology at $(1,1)$ so $\dim \ker H^0(g(1,1)) = r(\gamma-1)$. Let $v_1, \dotsc, v_{r(\gamma -1)}$ be a basis. Let $l \in H^0(\cO(a-1,b-1))$. Then $l\cdot v_i \in \ker H^0(g(a,b))$ hence $\dim \ker H^0(g(a,b)) \geq r(\gamma-1)a b$. Moreover 
\begin{align*}
\dim \left( \begin{array}{c}
\C^{r\gamma (a+1)b}\\
\C^{r \gamma a(b+1)}
\end{array}\right) - \bigg(r(\gamma-1)a b\bigg) &= r \gamma a b + r \gamma(a+b) + r a b\\
&\leq r \gamma a b + r \gamma(a+b) + r \gamma\\
&= (a+1)(b+1)r \gamma
\end{align*}
So no additional elements are forced into the kernel for dimension reasons. Hence the condition that $E^{univ}$ has natural cohomology at $(a,b)$ is that $H^0(g(a,b))$ has minimal kernel or equivalently maximal rank which is open. 

Now assume $a>0,b>0$ and $r a b - r \gamma > 0$ then the condition that $E^{univ}$ has natural cohomology at $(a,b)$ is that $H^0(g(a,b))$ is surjective or equivalently that $H^0(g(a,b))$ has maximal rank which is an open condition. 
\end{proof}

We now prove an analogue of proposition \ref{p:sign_flip}. Recall $p \colon \P^1 \x \P^1 \to \P^1$ is projection onto the first factor and $q \colon \P^1 \x \P^1 \to \P^1$ projects onto the second.
\begin{prop}\label{p:general_sign_flip}
Let $E$ be a vector bundle with $\chi(E(x,y)) = r x y - r \gamma$. Suppose $\chi(E(a,b))\leq 0$ and assume further that $E$ has natural cohomology at $(a,b)$ and at $(a \pm 1, b\pm 1)$.  
\begin{itemize}
\item[(a)] If $a>0,b>0$ and $\chi(E(a+1,b))>0$ then $R^1p_*E(a,b) = 0$ and $p_*E(a,b)$ is a vector bundle with natural cohomology.
\item[(b)] If $a>0,b>0$ and $\chi(E(a,b+1))>0$ then $R^1q_*E(a,b) = 0$ and $q_*E(a,b)$ is a vector bundle with natural cohomology.
\item[(c)] If $a<0,b<0$ and $\chi(E(a-1,b))>0$ then $p_*E(a,b) = 0$ and $R^1p_*E(a,b)$ is a vector bundle with natural cohomology.
\item[(d)] If $a<0,b<0$ and $\chi(E(a,b-1))>0$ then $q_*E(a,b) = 0$ and $R^1q_*E(a,b)$ is a vector bundle with natural cohomology.
\end{itemize}
\end{prop}
\begin{proof}
The argument is the same as the proof of proposition \ref{p:sign_flip}. For parts (c),(d) we prove the equivalent Serre dual statements:
\begin{itemize}
\item[(c$^\star$)]If $a>0,b>0$ and $\chi(E(a+1,b))>0$ then $R^1p_*E(a,b) = 0$ and $p_*E(a,b)$ is a vector bundle with natural cohomology.
\item[(d$^\star$)]If $a>0,b>0$ and $\chi(E(a,b+1))>0$ then $R^1q_*E(a,b) = 0$ and $q_*E(a,b)$ is a vector bundle with natural cohomology.
\end{itemize}
For example, to prove (a) we note that $E$ has only $H^1$ at $(a,b)$ and only $H^0$ at $(a+1,b)$. It follows that $R p_*E(a,b)$ has no torsion since by lemma \ref{l:kunneth} this would require $E$ to have nonzero $H^0$ at $(a,b),(a+1,b)$ or to have nonzero $H^1$ at $(a,b),(a+1,b)$. Next $R^1p_*E(a+1,b)$ has no cohomology so if it were nonzero then it would be $\cO(-1)^s$. But then twisting down we would get $0 \neq H^1(R^1p_*(a,b)) \subset H^2(E(a,b))$ which is a contradiction. Therefore $F = p_*E(a,b)$ is a vector bundle that has only $H^1$ and $F(1)$ has only $H^0$ so it must be that $F$ is a direct sum of $\cO(-2)$s and $\cO(-1)$s. Then (b),(c$^\star$),(d$^\star$) follow in the same way.
\end{proof}

As in the case of section \ref{s:gamma_leq_1}, to show a monad bundle $E$ has natural cohomology everywhere it will suffice to show it has natural cohomology at a finite number of twists. More precisely,  
\[
T_{E,+} =\{(a,b) \in \Z^2| a>0,b>0,\chi(E(a,b))\leq 0,\chi(E(a+1,b)>0 \mbox{or } \chi(E(a,b+1)) >0\}
\]
and similarly define
\[
T_{E,-} =\{(a,b) \in \Z^2| a<0,b<0,\chi(E(a,b))\leq 0,\chi(E(a-1,b)>0 \mbox{or } \chi(E(a,b-1)) >0\}
\]
The sets $T_{E,+},T_{E,-}$ are finite and we have
\begin{lemma}\label{l:finitely_many_conditions}
There exists $(f,g) \in Z^{(2,2)}_{(1,1)}$ such that if $E = E(f,g)$ is the associated monad vector bundle. 
then
\begin{itemize}
\item[(a)] $E$ has natural cohomology at $(0,n),(n,0)$ for $n \in \Z$.
\item[(b)] $E$ has natural cohomology at all $(a,b) \in T_{E,+},T_{E,-}$
\end{itemize}
\end{lemma}
\begin{proof}
Let $S = \{(0,\pm 1),(\pm 1,0)\} \cup T_{E,+}\cup T_{E,-}$. Then by proposition \ref{p:nat_cohom_at_(a,b)} there are non empty open subsets $U_{(a,b)} \subset Z^{(2,2)}_{(1,1)}$ such that if $(f,g) \in U_{(a,b)}$ then $E = E(f,g)$ has natural cohomology at $(a,b)$. Set
\[
U = \bigcap_{(a,b) \in S} U_{(a,b)},
\]
then we claim if $(f,g) \in U$ then $E = E(f,g)$ satisfies (a),(b). That $E$ satisfies (b) is built into the definition of $U$.

To prove (a) note that at $(0,0)$ the monad has only one nonzero term and trivially has natural cohomology. By the construction of $U$ we have that $E$ has natural cohomology at $(1,0),(0,1)$ and by the proof of lemma \ref{l:along_pos_axes} this implies $E$ has natural cohomology along $(n,0),(0,n)$ for $n\geq 0$. A Serre dual argument shows moreover that $E$ has natural cohomology along $(0,n),(n,0)$ for $n \leq 0$. 
\end{proof}

Finally we can prove 
\begin{thm}\label{thm:main}
Conjecture \ref{conj} holds: for any $P(x,y) = (x-\alpha)(y-\beta)-\gamma \in \Q[x,y]$ with $\gamma>0$ there exists a vector bundle $E$ with natural cohomology and Hilbert polynomial $\chi(E(a,b)) = \rank(E) P(a,b)$ for $\rank(E)$ sufficiently big.
\end{thm}
\begin{proof}
When $P \neq x y - \gamma$ the result is covered by \cite{Solis2017}. So we can assume $P = xy - \gamma$. When $\gamma \leq 1$ the result is covered by corollary \ref{c:gamma_leq_1}. Therefore we can assume $\gamma >1$.

By lemma \ref{l:finitely_many_conditions} we can find a monad bundle $E = E(f,g)$ such that $E$ has natural cohomology along the axes and at the finitely many places where $\chi(E(x,y))$ chages sign when $x$ or $y$ is shifted by $\pm 1$. 

Suppose $(a,b) \in \Z^2$ with $b \geq 0$. If $\chi(E(1,b))>0$ then because $\chi(E(0,b))<0$ we can apply the proof of proposition \ref{p:sign_flip} to see that $p_*E(0,b)$ is a vector bundle with natural cohomology hence so is $p_*E(a,b)$ by the projection formula thus $E$ has natural cohomology at $(a,b)$. On the other hand if $\chi(E(1,b))<0$ then we can find a minimal $m$ such that $\chi(E(m,b))\leq 0$ and $\chi(E(m+1,b))>0$. Then applying proposition \ref{p:general_sign_flip} we see that $p_*E(m,b)$ is a vector bundle with natural cohomology hence so is $p_*E(a,b)$. 

In a similar fashion we can show $E$ has natural cohomology at any $(a,b)$ with $a \geq 0$. Finally a Serre dual argument shows $E$ also has natural cohomology when either $a<0$ or $b<0$. Altogether this shows $E$ has natural cohomology everywhere. 
\end{proof}

\begin{rmk}\label{rmk:all cases}
Here we mention how to utilize the monads approach to handle the cases $P \neq x y - \gamma$. Given $P = (x - \alpha)(y - \beta) - \gamma$ we expect to get the desired bundle from a monad built up from $\cO(-1,-1)$, $\cO(0,-1)$, $\cO(-1,0)$, $\cO$. The various exponents and placement of the terms of the monad are determined by the sign value of $P(-1,-1)$, $P(0,-1)$, $P(-1,0)$, $P(0,0)$.  We can narrow down the possibilities by assuming $\alpha,\beta \in [n,n+1) \cap \Q$ for any $n \in \Z$. For example if $\alpha,\beta \in (-1,0]\cap \Q$ then $P(-1,0) = (-1-\alpha)(-\beta)-\gamma<0$ and similarly $P(0,-1)$ is always negative. But $P(-1,-1)$ and $P(0,0)$ can change sign. Provided that $P(0,0)$ and $P(-1,-1)$ are not both positive we get three cases:
\begin{align*}
0\to\cO(-1,-1)^{-rP(-1,-1)} \to &\begin{array}{c}
\cO(0,-1)^{-rP(0,-1)}\\
\cO(-1,0)^{-rP(-1,0)}
\end{array} \to \cO^{-rP(0,0)}\to 0\\
&\\
& \begin{array}{c}
\cO(0,-1)^{-rP(0,-1)}\\
\cO(-1,0)^{-rP(-1,0)}\\
\cO(-1,-1)^{rP(-1,-1)}
\end{array} \to \cO^{-rP(0,0)}\to 0\\
&\\
0\to \cO(-1,-1)^{-rP(-1,-1)} \to &\begin{array}{c}
\cO(0,-1)^{-rP(0,-1)}\\
\cO(-1,0)^{-rP(-1,0)}\\
\cO^{rP(0,0)}
\end{array} 
\end{align*}
where $r$ is chosen so all exponents are integral. 
But it can happen that either $P(0,0) = P(-1,-1) = 0$ or $P(0,0)$ and $P(-1,-1)$ are both positive. For example $P(x,y) = (x+\frac{1}{2}) (y+\frac{1}{2}) - \frac{1}{4}$ or $P(x,y) = (x+\frac{1}{2}) (y+\frac{1}{2}) - \frac{1}{5}$. In such cases one can apply an integral shift. In the previous two cases we can work with $P(x-1,y-1)$ and $P(x-1,y)$ and then evaluating at the appropriate value one can extract a valid monad.

We have not presented an extensive list of all the cases that arise and how to deal with them; this is manner of diligent bookkeeping. Once the appropriate monad is constructed one can show it has natural cohomology using the strategy presented here.
\end{rmk}

\bibliographystyle{plain} 
\bibliography{NatCohomP1P1}

\begin{thebibliography}{1}

\bibitem{MR509589}
Wolf Barth and Klaus Hulek.
\newblock Monads and moduli of vector bundles.
\newblock {\em Manuscripta Math.}, 25(4):323--347, 1978.

\bibitem{MR0498569}
V.~G. Drinfeld and Ju.~I. Manin.
\newblock Locally free sheaves on {$CP^{3}$} related to {Y}ang-{M}ills fields.
\newblock {\em Uspehi Mat. Nauk}, 33(3(201)):165--166, 1978.

\bibitem{EisenbudMR2810424}
David Eisenbud and Frank-Olaf Schreyer.
\newblock Boij-{S}\"oderberg theory.
\newblock In {\em Combinatorial aspects of commutative algebra and algebraic
  geometry}, volume~6 of {\em Abel Symp.}, pages 35--48. Springer, Berlin,
  2011.

\bibitem{MR564443}
R.~Hartshorne.
\newblock Algebraic vector bundles on projective spaces, with applications to
  the {Y}ang-{M}ills equation.
\newblock In {\em Complex manifold techniques in theoretical physics ({P}roc.
  {W}orkshop, {L}awrence, {K}an., 1978)}, volume~32 of {\em Res. Notes in
  Math.}, pages 35--44. Pitman, Boston, Mass.-London, 1979.

\bibitem{Solis2017}
P.~{Solis}.
\newblock {Hunting Vector Bundles on $\mathbb{P}^1 \times \mathbb{P}^1$}.
\newblock {\em ArXiv e-prints}, October 2017.

\end{thebibliography}

\end{document}